\documentclass[a4paper]{article}

\usepackage[utf8]{inputenc}
\usepackage[T2A]{fontenc}

\usepackage[english]{babel}

\usepackage{amsmath,amssymb,amsthm}
\usepackage{mathrsfs}
\usepackage{authblk}

\usepackage[left=2cm,right=2cm,top=2cm,bottom=2cm,bindingoffset=0cm]{geometry}
\usepackage{nccfoots}
\usepackage[hang,flushmargin]{footmisc} 
\usepackage{hyperref}

\usepackage{pgf,tikz}

\usetikzlibrary{calc,math}
\usetikzlibrary{angles,decorations.markings}

\tikzset{mydeco/.style={pic actions/.append code=\tikzset{postaction=decorate}}}

\newtheorem{theorem}{Theorem}

\def\lc{\left\lceil}   
\def\rc{\right\rceil}

\title{On set systems without singleton intersections}

\author{Danila Cherkashin}

\begin{document}

\maketitle

\begin{abstract}
Consider a family $\mathcal{F}$ of $k$-subsets of an ambient $(k^2-k+1)$-set such that no pair of $k$-subsets in $\mathcal{F}$ intersects in exactly one element.   
In this short note we show that the maximal size of such $\mathcal{F}$ is $\binom{k^2-k-1}{k-2}$ for every $k > 1$. 
\end{abstract}


\section{Introduction}

A large branch of combinatorics grows from the celebrated Erd{\H o}s--Ko--Rado theorem~\cite{erdos1961intersection}, which states that
a family of pairwise intersecting $k$-element subsets of an $n$-element set has size at most $\binom{n-1}{k-1}$ provided that $n \geq 2k$. 

For our aim it is convenient to define a \textit{Johnson graph} $J(n,k,t)$, whose vertices are $k$-element subsets of an $n$-element set and edges connect pairs of vertices with intersection $t$. An \textit{independent set} is a vertex subset of a graph such that it does not contain an edge.
Let $\alpha(G)$ stand for the size of a maximal independent set in a graph $G$.

In this language the statement of the Erd{\H o}s--Ko--Rado theorem is 
\[
\alpha (J[n,k,0]) = \binom{n-1}{k-1}
\]
for $n \geq 2k$. The problem of finding $\alpha (J[n,k,t])$ is known as Erd{\H o}s--S{\'o}s forbidden intersection problem.
The bibliography on this problem is wide and the proofs use very different techniques. Let us briefly provide the highlights.
Frankl and F{\"u}redi~\cite{frankl1985forbidding} used so-called $\Delta$-system method to show that 
$\alpha(J[n,k,t]) = \binom{n-t}{k-t}$ for $n > n_0(k)$ and $k \geq 2t+2$.

Another very general result was obtained by Frankl and Wilson~\cite{frankl1981intersection} by a rank bound: it gives 
$\alpha(J[n,k,t]) \leq \binom{n}{k-t-1}$ for $k > 2t$ and $k-t$ being a prime power.
This result has important applications to discrete geometry, see~\cite{kahn1993counterexample,raigorodskii2001borsuk}. 

Recently, Ellis, Keller and Lifshitz~\cite{ellis2024stability,keller2021junta} used junta-method to determine $\alpha(J[n,k,t])$ for 
$\varepsilon < k/n < 1/2 - \varepsilon$ and $n > n_0(t,\varepsilon)$. 
Kupavskii and Zakharov showed that $\alpha(J[n,k,t]) = \binom{n-t-1}{k-t-1}$ for $k > k_0$, $n = \lc k^\alpha \rc$, $t = \lc k^\beta \rc$, where $\alpha > 1$ and $1/2 > \beta > 0$ satisfy $\alpha > 1 + 2\beta$. They use spread approximation technique~\cite{alweiss2020improved}.

Also let us mention that $\alpha(J[n,4,1]) = \binom{n-2}{2}$ for $n \geq 9$, see Keevash--Mubayi--Wilson~\cite{keevash2006set}.

Our contribution is the following

\begin{theorem} \label{th:main}
    For every $k > 1$ one has
    \[
    \alpha(J[k^2-k+1,k,1]) = \binom{k^2-k-1}{k-2}.
    \]
\end{theorem}

Note that for $k > k_0$ Theorem~\ref{th:main} follows from the mentioned result of Kupavskii and Zakharov~\cite{kupavskii2024spread}.

\section{Tools}

\subsection{Johnson scheme and Bose--Mesner algebra}

The facts from this subsection can be found in books~\cite{godsil2016erdos} and~\cite{bannai2021algebraic}.

\textit{An associative scheme} is a pair $(V,\mathcal{R})$ consisting of a finite set $V$ and a family of non-empty binary relations $\mathcal{R}$ on $V$, and satisfying the following properties:
\begin{enumerate}
     \item sets $R \in \mathcal{R}$ form a partition $V^2$;
     \item the diagonal $\Delta(V)$ of the set $V^2$ is an element of $\mathcal{R}$;
     \item the set $\mathcal{R}$ is closed under the interchange of the first and second coordinates in $V^2$;
     \item for arbitrary relations $R, S, T \in \mathcal{R}$ numbers
     \[
     | v \in V : (u,v) \in R, (v,w) \in S|
     \]
     are the same for all $(u,w) \in T$.
\end{enumerate}

The classical \textit{Johnson scheme} is given by 
\[
V = \binom{[n]}{k}; \quad R_i := \{(v,u) \in V\times V : \langle v,u \rangle = k-i\}, \quad  i=0,\dots k.
\]

The Bose--Mesner algebra of the Johnson scheme is the algebra of $\binom{n}{k} \times \binom{n}{k}$ matrices, with entries defined by $A(x,y) := f(|x \cap y|)$. Since this is indeed algebra, these matrices are simultaneously diagonalizable.
A standard basis of the Bose--Mesner algebra is formed by matrices $B_{i}(x,y) := \binom{|x\setminus y|}{i}$, $i = 0,\dots, k$.
The eigenvalues of $B_i$ are given by
\[
\mu_j^{(i)} = (-1)^j \binom{k-j}{i-j} \binom{n-i-j}{k-j}.
\]
Then the machinery offers to represent any matrix $A$ from the Bose--Mesner algebra as $\sum b_iB_i$ and get it spectrum as 
\[
\lambda_j = \sum_{i=0}^k b_i \mu_j^{(i)}.
\]

Let $I$ be a subset of $\binom{[n]}{k}$ and $\chi_I$ stand for its characteristic vector. Denote by $c_i$, $i=0,\dots, k$ the coefficients in the decomposition of $\chi_I$ with respect to the common eigenspaces of the Bose--Mesner algebra.

\subsection{Hoffman bound}

The following celebrated theorem is widely known as Hoffman ratio bound~\cite{haemers2021hoffman}.

\begin{theorem}
    Let $A$ be a pseudoadjacency matrix of a $d$-regular $N$-vertex graph $G$ with non-negative entries. Then 
\begin{equation} \label{eq:Hoffmanita}
\alpha(G) \leq N \frac{-\lambda_{min}}{d - \lambda_{min}},
\end{equation}
where $\alpha(G)$ is the independence number of $G$.
\end{theorem}

The proof is a one-line collection of several simple observations. Let $I$ be any independent set in $G$ and $\chi_I$ stand for its characteristic vector. Then
\begin{equation} \label{eq:main}
0 = (A\chi_I,\chi_I) = \sum_{i=1}^N a_i^2 \lambda_i \geq a_1^2 d + (a_2^2 + \dots + a_N^2) \lambda_{min} = \frac{|I|^2}{N} d + \left (|I| - \frac{|I|^2}{N}  \right) \lambda_{min},
\end{equation}
where $a_i$ are the coefficients in the decomposition of $\chi_I$ in the eigenbasis of $A$.
Here we use that a spectral radius of a $d$-regular graph is $d$ and it is achieved at the all-unit vector.
Also in the case of an edge-transitive graph Hoffman bound coincides with Lov{\'a}sz theta-bound~\cite{lovasz1979shannon}.

\section{The proof}

\begin{proof}[Proof of Theorem~\ref{th:main}]
  An example of an independent set of size $\binom{k^2-k-1}{k-2}$ is given by a collection of all sets, containing elements 1 and 2.

Let $A$ be the adjacency matrix of the Johnson graph $J(n,k,1)$.  Clearly, this matrix is given by $f(1) = 1$, $f(0) = f(2) = \dots = f(k) = 0$. It is straightforward to check that the coefficients in the standard basis of the Bose--Mesner algebra are the following
    \[
    b_0 = b_1 = \dots = b_{k-2}=0, \quad b_{k-1} = 1, \quad b_k = -k
    \]
    and
    \[
    \lambda_0 = -k\binom{n-k}{k} + k\binom{n-k+1}{k} = k \binom{n-k}{k-1},
    \]
    \[
    \lambda_1 = k\binom{n-k-1}{k-1} - (k-1)\binom{n-k}{k-1}, \quad \lambda_2 = -k\binom{n-k-2}{k-2} + (k-2)\binom{n-k-1}{k-2}.
    \]
For $n = k^2-k+1$ we have
\[
\lambda_1 = \lambda_2 = -\frac{1}{k-1} \binom{n-k}{k-1} < 0 
\]
and
\[
\lambda_3 = k\binom{n-k-3}{k-3} - (k-3)\binom{n-k-2}{k-3} = \frac{2k^2-3k-3}{k^2-3k+2} \binom{n-k-3}{k-3} > 0.
\]
Also  $|\lambda_4|, |\lambda_5|, \dots, |\lambda_k| \leq k \binom{n-k}{k-3}$, hence $\lambda_1 = \lambda_2$ are the smallest eigenvalues.

Now the upper bound follows from the Hoffman bound~\eqref{eq:Hoffmanita}:
\[
\frac{-\lambda_{min}}{d - \lambda_{min}} = \frac{1}{k^2-k+1} = \frac{\binom{n-2}{k-2}}{\binom{n}{k}}.
\]
    
\end{proof}

\section{Discussion}

Let us briefly discuss the sporadic nature of the result. In all solved cases maximal independent sets form designs or juntas.
Hoffman bound is tight when the corresponding characteristic vector belongs to maximal and minimal eigenspaces and it seems difficult to modify the method in other cases.
The maximal eigenspace is always unique and corresponds to the all-unit vector.
For $t \geq 1$ the characteristic vectors of all known examples belong to more than two eigenspaces, so several minimal eigenvalues should coincide in order to use Hoffman bound. Summing up, it seems that the only case is $t=1$, in which we have an example with the characteristic vector in the first three eigenspaces.
So we need $\lambda_1 = \lambda_2$ which implies $n = k^2 - k + 1$.

\subsection{Finite projective planes}

If $k-1$ is a prime power, then one can prove the upper bound in Theorem~\ref{th:main} combinatorially. Since Johnson graphs are vertex-transitive (moreover they are edge-transitive), one has
    \[
    w(J) \cdot \alpha(J) \leq |V(J)|.
    \]
    In our case $k^2 - k + 1$ is the size of a projective plane over $GF(k-1)$, and so $w(J[k^2-k+1]) \geq k^2-k+1$.
    This immediately implies the bound.

    However for a composite $k-1$ the corresponding construction may not exist.
    A major negative result is a celebrated Bruck--Ryser theorem~\cite{bruck1949nonexistence} which states that if $n$ is a positive integer of the form $4k + 1$ or $4k + 2$ and $n$ is not equal to the sum of two integer squares, then $n$ does not occur as the order of a finite plane.  A widely known conjecture is that the order of a finite plane is always a prime power. Also, the non-existence of a finite plane of order 10 was proven by Lam~\cite{lam1991search}.

\subsection{Uniqueness}

For $k = 3$ the graph $J(7,3,1)$ has a maximal independent set of another structure, namely 
\[
\{\{1,2,3\},\{1,2,4\},\{1,3,4\},\{2,3,4\},\{5,6,7\}\}.
\]
For other values of $k$ it seems very likely that every maximal independent set of $J(k^2-k+1,k,1)$ forms a family with two elements in common.
However we are not able to prove it. Following the proof of Theorem~\ref{th:main}, an independent set $I$ of the maximal size satisfies the equality in~\eqref{eq:main} and thus $\chi_I$ belongs to the zeroth, the first and the second eigenspaces.
Main obstacle in our attempt is a relatively complicated structure of the second eigenspace.

\paragraph{Acknowledgements.} The research is supported by Bulgarian NSF grant KP-06-N72/6-2023.

\bibliography{main}
\bibliographystyle{plain}

\end{document}